\documentclass[a4paper,11pt]{amsart}
\usepackage{tipa}
\usepackage{txfonts}
\usepackage{amssymb}
\usepackage{amsmath}
\usepackage{mathrsfs}
\usepackage{amsmath,amssymb,amsthm,latexsym,amscd,mathrsfs}
\usepackage{indentfirst}
\usepackage{graphicx}
\usepackage{booktabs}
\usepackage{array}
\usepackage{chemarrow}
\newcommand{\ROM}[1]{\mathrm{\uppercase\expandafter{\romannumeral#1}}}
\theoremstyle{definition}

\newtheorem{thm}{Theorem}[section]

\newtheorem{rem}{Remark}[section]
\newtheorem{prop}{Proposition}[section]

\newtheorem{ack}{Acknowledgements}   
\makeatletter


\title[Ricci curvature of double manifolds via isoparametric
foliations]{\textbf{Ricci curvature of double manifolds via isoparametric
foliations}}

\author[C. K. Peng]{ChiaKuei Peng}\address{School of Mathematical Sciences, University of Chinese Academy of Sciences , Beijing 100049, PR China}\email{pengck@ucas.ac.cn}

\author[C. Qian]{Chao Qian}\address{School of Mathematics and Statistics, Beijing Institute of Technology, Beijing 100081, PR China}
\email{6120150035@bit.edu.cn}
\thanks{}

\thanks{The second author is the corresponding author.}

\date{}
\begin{document}

\maketitle
\begin{abstract}
Given a closed manifold $M$ and a vector bundle $\xi$ of rank $n$ over $M$, by gluing two copies of the
disc bundle of $\xi$, we can obtain a closed manifold $D(\xi, M)$, the so-called double manifold.

In this paper,
%we will study Ricci curvature properties of double manifolds with isoparametric foliations.
we firstly prove that each sphere bundle $S_r(\xi)$ of radius $r>0$ is an isoparametric hypersurface in the total space of $\xi$ equipped with a connection metric, and for $r>0$ small enough, the induced metric of $S_r(\xi)$ has positive Ricci curvature under the additional assumptions that $M$ has a metric with positive Ricci curvature and $n\geq3$.

As an application, if $M$ admits a metric with positive Ricci curvature and $n\geq2$, then we construct a metric with positive Ricci curvature on $D(\xi, M)$. Moreover, under the same metric, $D(\xi, M)$ admits a natural isoparametric foliation.

For a compact minimal isoparametric hypersurface $Y^{n}$ in $S^{n+1}(1)$, which separates $S^{n+1}(1)$ into
$S^{n+1}_+$ and $S^{n+1}_-$, one can get double manifolds $D(S^{n+1}_+)$ and $D(S^{n+1}_-)$. Inspired by Tang, Xie and Yan's work on scalar curvature of such manifolds with isoparametric foliations(cf. \cite{TXY12}), we study Ricci curvature of them with isoparametric foliations in the last part.

%Finally, we can improve
%Theorem 1.1 in \cite{TXY12}
%[Tang, Xie, and Yan, Comm. Anal. Geom., \textbf{20}(2012), 989--1018]
%by proving that the double manifold, constructed from
%an isoparametric hypersurface with four principal curvatures in a unit sphere, admits a metric with positive Ricci curvature(if the fundamental group of the double manifold is not finite, then the metric only has nonnegative Ricci curvature), and also an isoparametric foliation.
\end{abstract}

%------------------------------------------------------------------------------------------------------------------------
\section{Introduction}
Let $N$ be a connected complete Riemannian manifold. A non-constant smooth function $f:N\rightarrow \mathbb{R}$ is called \emph{transnormal}, if there exists a smooth function $b: \mathbb{R} \rightarrow \mathbb{R}$ such that the gradient of $f$ satisfies
$|\nabla f|^2=b(f)$. Moreover, if there exists another function $a: \mathbb{R} \rightarrow \mathbb{R}$ such that the Laplacian of $f$ satisfies $\triangle f=a(f)$, then $f$ is said to be \emph{isoparametric}. The two equations mean that regular level hypersurfaces are parallel and all have constant mean curvature. Each regular level hypersurface is said to be an isoparametric hypersurface. In \cite{Wa87}, Wang proved that singular level sets are also smooth
submanifolds, the so-called \emph{focal submanifolds}. Recall that an \emph{isoparametric foliation} on a Riemannian manifold $N$ is defined by the whole family of regular level hypersurfaces
together with focal submanifolds of an isoparametric function(cf. \cite{Wa87}, \cite{GT13} and \cite{QT15}).
Equivalently, an isoparametric foliation is a singular Riemannian foliation
all of whose regular leaves are hypersurfaces with constant mean curvature(cf. \cite{GQ15}). For instance, the orbits of
an isometric cohomogeneity one action on $N$ form an isoparametric foliation,
the so-called \emph{homogeneous} isoparametric foliation. Hence, isoparametric foliations can be regarded as a geometric generalization of cohomogeneity one actions.

E. Cartan was the pioneer who made a comprehensive study of isoparametric hypersurfaces in real space forms, especially
in the unit spheres. Particularly, he proved that isoparametric hypersurfaces in real space forms are equivalently hypersurfaces with constant principal curvatures and obtained the classification results for the Euclidean and hyperbolic cases. For the spherical case, if the number of distinct
principal curvatures is no more than $3$, then he showed that the isoparametric hypersurfaces must be homogeneous. Given an isoparametric hypersurface $Y^n$ in $S^{n+1}(1)$, let $\eta$ be a unit normal vector field, $g$ the number of distinct principal curvatures, $\cot{\theta_{\alpha}}$ ($\alpha=1, 2, \cdots, g; 0<\theta_1<\theta_2<\cdots<\theta_g<\pi$) the principal curvatures with respect to $\eta$ and $m_{\alpha}$ the multiplicity of $\cot{\theta_{\alpha}}$. Under this setting, H. F. M\"{u}nzner \cite{Mu80} proved that
$\theta_{\alpha}=\theta_1+\frac{\alpha-1}{g}\pi$, $m_{\alpha}=m_{\alpha+2}
(\mathrm{indices~mod}~g)$, and there exists
a homogeneous polynomial $F: \mathbb{R}^{n+2}\rightarrow \mathbb{R}$ of degree $g$,
the so-called\emph{ Cartan-M\"{u}nzner polynomial}, satisfying
\begin{equation}\label{ab}
\left\{ \begin{array}{ll}
|\nabla^E F|^2= g^2r^{2g-2}, r=|x|,\nonumber\\
~~~~\triangle^E F~~=\frac{m_2-m_1}{2}g^2r^{g-2},\nonumber
\end{array}\right.
\end{equation}
where $m_1$ and $m_2$ are the two multiplicities, and $\nabla^E, \triangle^E$ are Euclidean gradient and Laplacian, respectively. It follows that $f=F|_{S^{n+1}(1)}$ is an isoparametric function on $S^{n+1}(1)$ with $\mathrm{Im}f=[-1, 1]$, and $M_{\pm}=f^{-1}(\pm)$ are two focal submanifolds
with codimension $m_1+1$ and $m_2+1$ in $S^{n+1}(1)$ respectively. Moreover, based on this global structure, he obtained the splendid result that
$g=1, 2, 3, 4$ or $6$.
Owing to E. Cartan and H. F. M\"{u}nzner, the classification of isoparametric hypersurfaces with $4$ or $6$ principal curvatures
in unit spheres is an intriguing problem in submanifold geometry. Up to now, isoparametric hypersurfaces with $4$ principal curvatures in unit spheres have been classified except for one case(cf. \cite{CCJ07}, \cite{Im08} and
\cite{Ch13}). For recent progress and application in this subject, we refer to
\cite{GX10}, \cite{TY13}, \cite{QT14} and \cite{TY15}.

One of the fundamental problems in Riemannian geometry is to investigate Riemannian manifolds
with special curvature properties, especially manifolds with positive scalar, Ricci or sectional
curvatures. For the scalar curvature case, due to the surgery theory of Schoen-Yau \cite{SY79} and Gromov-Lawson \cite{GL80}, simply connected
manifolds with positive scalar curvature are well understood. Motivated by Schoen-Yau
and Gromov-Lawson surgery theory, Tang, Xie and Yan in \cite{TXY12}
obtained the following theorem based on the theory of isoparametric hypersurfaces in unit spheres.
\begin{thm}(\cite{TXY12})
Let $Y^{n}$ be a compact minimal isoparametric hypersurface in $S^{n+1}(1)$, $n\geq 3$, which separates $S^{n+1}(1)$ into
$S^{n+1}_+$ and $S^{n+1}_-$. Then each of the doubles $D(S^{n+1}_+)$ and $D(S^{n+1}_-)$ admits a metric of positive
scalar curvature. Moreover, there is still an isoparametric foliation in $D(S^{n+1}_+)$ (or $D(S^{n+1}_-)$).
\end{thm}

In this paper, we will study Ricci curvature properties of
double manifolds $D(S^{n+1}_+)$ and $D(S^{n+1}_-)$ together with isoparametric foliations. We start
with a general setting. Let $\xi$ be a vector bundle of rank $n$ over a closed manifold $M$. Given two copies of the disc bundle $B(\xi)$ over $M$ with boundary the sphere bundle $S(\xi)$, we define the \emph{double manifold} $D(\xi, M):=B(\xi)\bigsqcup_{\mathrm{id}} B(\xi)$, where $\mathrm{id}: S(\xi)\rightarrow S(\xi)$ is the identity map. For instance,
$D(S^{n+1}_+)\cong D(\nu(M_+), M_+)$ and $D(S^{n+1}_-)\cong D(\nu(M_-), M_-)$, where $\nu(M_+)$ and $\nu(M_-)$ are the
normal bundles over the focal submanifolds $M_+$ and $M_-$ of isoparametric
hypersurface $Y^n$ in $S^{n+1}(1)$,
respectively.

To consider the Ricci curvature properties of double manifolds, we firstly obtain the following result.
\begin{thm}\label{Positive Ricci}
Let $\xi$ be a Riemannian vector bundle of rank $n$ over a closed Riemannian manifold $(M, ds^2_M)$. Given a connection metric on $E$, then each sphere bundle $S_r(\xi)$ of radius $r>0$ is an isoparametric hypersurface in $E$. Moreover, assume $(M, ds^2_M)$ has positive Ricci curvature and $n\geq 3$, then, for sufficient small $r>0$, the sphere bundle $S_r(\xi)$ of radius $r>0$ with the induced metric in $E$ has positive Ricci curvature.
\end{thm}

\begin{rem}\label{Positive scalar}
Actually, for sufficient small $r>0$, the sphere bundle $S_r(\xi)$ of radius $r$ with the induced metric in $E$ has positive scalar curvature without the assumption that $(M, ds^2_M)$ has positive Ricci curvature.
\end{rem}

\begin{rem}
Theorem \ref{Positive Ricci} and Remark \ref{Positive scalar} generalize Theorem $1$ and Theorem $3$ in \cite{KS00}(also see \cite{Na79}).
\end{rem}

As an application of Theorem \ref{Positive Ricci}, we can prove
\begin{thm}\label{main}
Assume $n\geq2$, and the closed manifold $M$ has a metric with positive Ricci curvature. Then the double manifold $D(\xi, M)$ admits
a metric with positive Ricci curvature, and meanwhile a natural isoparametric foliation.
\end{thm}
\begin{rem}
By the assumption that $M$ has a metric with positive Ricci curvature, it follows from Bonnet-Myers theorem that $\pi_1M$ is finite. Then, by Van Kampen's theorem,
$\pi_1D(\xi, M)$ is isomorphic to $\pi_1M$, and is also finite.
\end{rem}

At last, combining with Theorem A in \cite{GZ02} and Theorem \ref{main}, we can infer the following result, which improves Theorem 1.1 in \cite{TXY12}.
\begin{thm}\label{isoparametric case}
Let $Y^n$ be a compact isoparametric hypersurface in a unit sphere $S^{n+1}(1)$ with
$4$ distinct principal curvatures and multiplicities $(m_1, m_2)$(assume $m_1\leq m_2$), and $M_{\pm}$ the focal submanifolds of $Y^n$ with codimension $m_1+1$ and $m_2+1$ in $S^{n+1}(1)$ respectively.

\noindent (1).  $(m_1, m_2)=(1, k)$: The double manifold $D(S^{n+1}_+)$, associated with $M_+$ which is diffeomorphic to the Stiefel manifold $V_{k+2, 2}$, admits
a metric with positive Ricci curvature and a homogeneous isoparametric foliation.
However, the double manifold $D(S^{n+1}_-)$, associated with $M_-$ which is
diffeomorphic to $(S^1\times S^{k+1})/\mathbb{Z}_2$, only admits a metric
with nonnegative Ricci curvature and a homogeneous isoparametric foliation.

\noindent (2).  $2\leq m_1\leq m_2$: Both of the resulting double manifolds $D(S^{n+1}_+)$ and $D(S^{n+1}_-)$ admit metrics with positive Ricci curvature,
and also natural isoparametric foliations.
\end{thm}

\begin{rem}
For case $(m_1, m_2)=(1, k)$, by Van Kampen's theorem, $\pi_1 D(S^{n+1}_-)=\pi_1 (S^1\times S^{k+1})/\mathbb{Z}_2=\mathbb{Z}$(cf. \cite{CR85}, \cite{QT14a}). Hence, according to Bonnet-Myers theorem, $D(S^{n+1}_-)$ can not admit a metric with positive Ricci curvature in this case.
\end{rem}

\begin{rem}
The natural isoparametric foliations constructed in Theorem \ref{isoparametric case} are foliated diffeomorphic to the ones constructed
in Theorem 1.1 of \cite{TXY12}. For the definition of foliated diffeomorphism, we refer to \cite{Ge14} and \cite{GQ15}.
\end{rem}
%\begin{rem}
%Based on Remark 1.1, one can give another proof for Theorem 1.1 in \cite{TXY12}, without using Schoen-Yau and Gromov-Lawson surgery theory.
%\end{rem}

\section{Geometry of vector bundles and associated sphere bundles}
In this paper, we shall make use of the following convention on the ranges of indices:
$$1\leq A, B, C,...\leq m+n, 1\leq i, j, k,...\leq m,$$
$$ m+1\leq \alpha, \beta, \gamma,...\leq m+n, m+1\leq a, b, c,...\leq m+n-1.$$

Generally, let $(N, ds^2_N)$ be a Riemannian manifold of dimension $m+n$. Choose a local orthonormal frame $e_1,\cdots, e_{m+n}$, and let
$\theta_1,\cdots, \theta_{m+n}$
be the dual frame such that $ds^2_N=\sum_{A=1}^{m+n}(\theta_{A})^2$. Then we have the first and second structural equations:
\begin{equation}\label{ab}
\left\{ \begin{array}{ll}
d\theta_A=\theta_{AB}\wedge \theta_B, \theta_{AB}+\theta_{BA}=0,\nonumber\\
d\theta_{AB}=\theta_{AC}\wedge\theta_{CB}-\Theta_{AB},\nonumber
\end{array}\right.
\end{equation}
where $\theta_{AB}$ are the connection $1$-forms, and $\Theta_{AB}=\frac{1}{2}K_{ABCD}\theta_C\wedge \theta_D$
are the curvature $2$-forms with $K_{ABCD}$ the components of the curvature tensor.

In this section, we will consider the geometry of vector bundles and sphere bundles. Let $(M, ds_M^2)$ be a $m$-dimensional closed Riemannian manifold, and $\xi$ a Riemannian vector bundle of rank $n$ over $M$ with total space
$E$ and projection
$\pi: E\rightarrow M$. For our purpose, we fix a metric
$\langle\cdot,\cdot\rangle$ on $\xi$ and choose a connection $D$ compatible with $\langle\cdot,\cdot\rangle$.

Let $U$ be an open neighborhood in $M$, and $ds_M^2|_U=\sum_{i=1}^{m}(\omega_i)^2$,
where $\omega_1,..., \omega_m$ are $1$-forms on $U$. Then the structural equations of $(M, ds^2_M)$ are given by
\begin{equation}\label{ab}
\left\{ \begin{array}{ll}
d\omega_i=\omega_{ij}\wedge \omega_j, \omega_{ij}+\omega_{ji}=0,\nonumber\\
d\omega_{ij}=\omega_{ik}\wedge\omega_{kj}-\Omega_{ij},\nonumber
\end{array}\right.
\end{equation}
where $\omega_{ij}$ are connection $1$-forms, and  $\Omega_{ij}=\frac{1}{2}R_{ijkl}\omega_k\wedge \omega_l$ are the curvature $2$-forms with
$R_{ijkl}$ the components of the curvature tensor.

Assume $U$ is small enough such that $\xi|_U$ is trivial, we can choose orthonormal cross sections
$e_{m+1},...,e_{m+n}$ on $U$ for $\xi$, i.e. $\langle e_{\alpha},e_{\beta}\rangle=\delta_{\alpha\beta}$, and $\Psi_U:U\times \mathbb{R}^n\rightarrow \pi^{-1}(U)$ is the local trivialization given by
$\Psi_U(p, v_1,... , v_n)=\sum_{\alpha=m+1}^{m+n}v_{\alpha}e_{\alpha}(p)$.
Then $De_{\alpha}=\omega_{\alpha\beta}e_{\beta},$
where $(\omega_{\alpha\beta})$ is the matrix expression of the connection $D$. Note that $D$ is compatible with
$\langle\cdot,\cdot\rangle$ means that $(\omega_{\alpha\beta})$ is skew-symmetric. The curvature $2$-forms $\Omega_{\alpha\beta}$ of the connection $D$ are determined by $d\omega_{\alpha\beta}=\omega_{\alpha\gamma}\wedge\omega_{\gamma\beta}-
\Omega_{\alpha\beta}$, where $\Omega_{\alpha\beta}=\frac{1}{2}R_{\alpha\beta i j}\omega_i\wedge\omega_j$.

Define $\theta_i:=\pi^*\omega_i$ and $\theta_{\alpha}:=dv_{\alpha}
+v_{\beta}\omega_{\beta\alpha}$.
Then $ds^2_{\xi}=(\theta_i)^2+(\theta_{\alpha})^2$
is a well-defined Riemannian metric on $E$, i.e., the so-called \emph{connection metric}(cf. \cite{Je73}, \cite{Na79}).
For convenience, we give the formulae of connection $1$-forms $\theta_{AB}$ and curvature
$2$-forms $\Theta_{AB}$ of the connection metric in the following by making use of the moving frame method.
\begin{prop}\label{connection metric}
The connection $1$-forms $\theta_{AB}$ are determined by
\begin{eqnarray}
\theta_{ij}&=&\pi^{*}\omega_{ij}+A_{ij\alpha}\theta_{\alpha},\nonumber\\
\theta_{i\alpha}&=&A_{ij\alpha}\theta_j,\nonumber\\
\theta_{\alpha\beta}&=&\omega_{\alpha\beta},\nonumber
\end{eqnarray}
and the curvature $2$-forms $\Theta_{AB}$ are given by
\begin{eqnarray}
\Theta_{ij}&=&\Omega_{ij}-A_{ik\alpha}A_{jl\alpha}\theta_k\wedge\theta_l
-A_{ij\alpha}A_{kl\alpha}\theta_k\wedge\theta_l
-\frac{1}{2}R_{\alpha\beta ij,k}v_{\beta}\theta_{k}\wedge\theta_{\alpha}+\frac{1}{2}R_{\alpha\beta ij}\theta_{\alpha}\wedge\theta_{\beta}\nonumber\\
&&+A_{ik\alpha}A_{kj\beta}\theta_{\alpha}
\wedge\theta_{\beta},
\nonumber\\
\Theta_{i\alpha}&=&\frac{1}{2}R_{\alpha\beta ij,k}v_{\beta}\theta_{j}\wedge\theta_{k}+\frac{1}{2}R_{\alpha\beta ij}\theta_{j}\wedge\theta_{\beta}-A_{ik\beta}A_{kj\alpha}\theta_j
\wedge\theta_{\beta},
\nonumber\\
\Theta_{\alpha\beta}&=&-A_{ij\alpha}A_{ik\beta}\theta_j\wedge\theta_k
+\Omega_{\alpha\beta},\nonumber
\end{eqnarray}
where $A_{ij\alpha}=\frac{1}{2}R_{\alpha\beta ij}v_{\beta}$, and
$R_{\alpha\beta ij,k}$ is defined by $R_{\alpha\beta ij,k}\omega_k:=
dR_{\alpha\beta ij}+R_{\gamma\beta ij}\omega_{\gamma\alpha}+R_{\alpha\gamma ij}\omega_{\gamma\beta}+R_{\alpha\beta lj}\omega_{li}+R_{\alpha\beta il}\omega_{lj}$.
\end{prop}
\begin{proof}

%\begin{eqnarray}
%d\theta_{i}&=&\theta_{ij}\wedge\theta_{j}+\theta_{i\alpha}\wedge\theta_{\alpha},\nonumber\\
%d\theta_{\alpha}&=&\theta_{\alpha %i}\wedge\theta_{i}+\theta_{\alpha\beta}\wedge\theta_{\beta},\nonumber
%\end{eqnarray}

%\begin{eqnarray}
%d\theta_{ij}&=&\theta_{ik}\wedge\theta_{kj}+\theta_{i\alpha}\wedge\theta_{\alpha %j}-\Theta_{ij},\nonumber\\
%d\theta_{i\alpha}&=&\theta_{ij}\wedge\theta_{j\alpha}+\theta_{i\beta}\wedge\theta_{\beta %\alpha}-\Theta_{i\alpha},\nonumber\\
%d\theta_{\alpha\beta}&=&\theta_{\alpha %i}\wedge\theta_{i\beta}+\theta_{\alpha\gamma}\wedge\theta_{\gamma\beta}
%-\Theta_{\alpha\beta}.\nonumber
%\end{eqnarray}
By using the first and second structural equations and a direct computation.
\end{proof}

As a consequence, we can obtain the components of curvature tensor, which are defined by $\Theta_{AB}:=\frac{1}{2}\Theta_{ABCD}\theta_C\wedge\theta_D$.
\begin{prop}
The components of the curvature $2$-forms $\Theta_{AB}$ are given by
\begin{eqnarray}
\Theta_{ijkl}&=&R_{ijkl}-A_{ik\alpha}A_{jl\alpha}+A_{il\alpha}A_{jk\alpha}
-2A_{ij\alpha}A_{kl\alpha},\nonumber\\
\Theta_{ijk\alpha}&=&-\frac{1}{2}R_{\alpha\beta ij,k}v_{\beta},\nonumber\\
\Theta_{ij\alpha\beta}&=&R_{\alpha\beta ij}+A_{ik\alpha}A_{kj\beta}-A_{ik\beta}A_{kj\alpha},\nonumber\\
\Theta_{i\alpha j\beta}&=&\frac{1}{2}R_{\alpha\beta ij}-A_{ik\beta}A_{kj\alpha},\nonumber\\
\Theta_{i\alpha\beta\gamma}&=&0,
~~\Theta_{\alpha\beta\gamma\sigma}~~~=~~~0.\nonumber
\end{eqnarray}
\end{prop}
%we choose $\varphi_{m+n}=dr, \varphi_i=\theta_{i}$ and $\varphi_{a}$ such that
%$$\sum_{a=m+1}^{m+n-1}\varphi_{a}^2=\sum_{\alpha=m+1}^{m+n}(du_{\alpha}
%+\sum_{\beta=m+1}^{m+n}u_{\beta}\omega_{\beta\alpha})^2.$$
%Then, by $d\varphi_{m+n}=0$ and the first structural equations of the moving frame $\{\varphi_{i}, \varphi_{a}, \varphi_{m+n}\}$, we have the connection forms
%$\varphi_{m+ni}=0, \varphi_{m+na}=\frac{1}{r}\varphi_{a}$. And the result follows.

Define $S_r(\xi):=\{(p,v)\in E~|~v\in \pi^{-1}(p), \langle v,v\rangle=r^2\}$, the associated
sphere bundle of radius $r>0$. At present, we are in the position to prove Theorem \ref{Positive Ricci}.

\noindent\textbf{Proof of Theorem \ref{Positive Ricci}:}
\begin{proof}
At first, observe that the function $f: E\rightarrow \mathbb{R}, (p,v)\mapsto \langle v, v\rangle$ is a transnormal function(cf. Theorem 2.2 in \cite{QT15}).
Next, we need to compute the principal curvatures of $S_r(\xi)\subset E$ for any $r> 0$. On $\pi^{-1}(U)$, $r^2=\sum_{\alpha}v_{\alpha}^2$. Define $u_{\alpha}:=\frac{v_{\alpha}}{r}$.
For $r>0$, $\eta=u_{\alpha}\frac{\partial}{\partial v_{\alpha}}$ is a natural
global normal vector field of $S_{r}(\xi)\subset E$.  In order to study the
extrinsic geometry of $S_r(\xi)$, we introduce the orthogonal transformation $P: \mathbb{R}^n\rightarrow \mathbb{R}^n$ defined by
$$(P_{\alpha\beta})=\left(
\begin{array}{cc}
I_{n-1}-\frac{x^Tx}{1-x_n} & x^T \\
x & x_n\\
\end{array}\right),$$
where $x=(u_{m+1},\ldots, u_{m+n-1})$ and $x_n=u_{m+n}$.
%which satifies $P=P^T$ and $P^2=I_n$.
Choose the moving frame $\{\varphi_{i}, \varphi_{a}, \varphi_{m+n}\}$ by
$\varphi_{i}=\theta_i$ and $\varphi_{\alpha}=P_{\alpha\beta}\theta_{\beta}$. Particularly, $\varphi_{m+n}=u_{\alpha}\theta_{\alpha}=dr$. Therefore, $\{\varphi_{i}, \varphi_{a}, \varphi_{m+n}\}$ is exactly the adapted moving frame with respect to $S_r(\xi)\subset E$. The associated connection $1$-forms $(\varphi_{AB})$ are given by
$$(\varphi_{AB})=\left(
\begin{array}{cc}
\varphi_{ij} & \varphi_{i\alpha} \\
\varphi_{\alpha i} & \varphi_{\alpha\beta}\\
\end{array}\right)=
\left(
\begin{array}{cc}
\theta_{ij} & \theta_{i\beta}P_{\beta\alpha} \\
P_{\alpha\beta}\theta_{\beta i} & dP_{\alpha\gamma}P_{\gamma\beta}+P_{\alpha\gamma}\theta_{\gamma\delta}P_{\delta\beta}\\
\end{array}\right).$$
It follows that $\varphi_{im+n}=0$ and $\varphi_{am+n}=-\frac{\varphi_{a}}{r}$. Therefore, for each $r>0$, $S_r(\xi)$ is an isoparametric hypersurface in $E$ with constant principal curvatures $0$ and $-\frac{1}{r}$ with multiplicities $m$ and $n$, respectively.

Now, under the assumption $(M, ds^2_M)$ has positive Ricci curvature and $n\geq3$, we will prove the induced metric of $S_r(\xi)\subset (E, ds^2_{\xi})$ has positive Ricci curvature for $r>0$ small enough. To compute the Ricci curvatures of $S_r(\xi)$, choose the moving frame $\{\varphi_{i}, \varphi_{a}, \varphi_{m+n}\}$ of $(E, ds^2_{\xi})$, and define $\Phi_{AB}:=\frac{1}{2}K_{ABCD}\varphi_{C}\wedge\varphi_{D}$ for $1\leq A, B, C, D\leq m+n$, where $\Phi_{AB}$ are curvature $2$-forms of $\{\varphi_{i}, \varphi_{a}, \varphi_{m+n}\}$. Then
$$(\Phi_{AB})=\left(
\begin{array}{cc}
\Phi_{ij} & \Phi_{i\alpha} \\
\Phi_{\alpha i} & \Phi_{\alpha\beta}\\
\end{array}\right)=
\left(
\begin{array}{cc}
\Theta_{ij} & \Theta_{i\beta}P_{\beta\alpha} \\
P_{\alpha\beta}\Theta_{\beta i} & P_{\alpha\gamma}\Theta_{\gamma\delta}P_{\delta\beta}\\
\end{array}\right).$$
Define $\psi_i=\varphi_i|_{S_r(\xi)}, \psi_a=\varphi_a|_{S_r(\xi)}$, then $\{\psi_i, \psi_a\}$ is a moving frame for $S_r(\xi)$. Let $\psi_{ij}, \psi_{ia}, \psi_{ab}$ be the associated connection $1$-forms, and $\Psi_{ij}, \Psi_{ia}, \Psi_{ab}$ the associated curvature $2$-forms respectively. By the Gauss equations of $S_r(\xi)\subset E$, it follows that
\begin{eqnarray}
\Psi_{ij}&=&\Phi_{ij}|_{S_r(\xi)}+\varphi_{im+n}\wedge\varphi_{jm+n}|_{S_r(\xi)},\nonumber\\
\Psi_{ia}&=&\Phi_{ia}|_{S_r(\xi)}+\varphi_{im+n}\wedge\varphi_{am+n}|_{S_r(\xi)},\nonumber\\
\Psi_{ab}&=&\Phi_{ab}|_{S_r(\xi)}+\varphi_{am+n}\wedge\varphi_{bm+n}|_{S_r(\xi)}.\nonumber
\end{eqnarray}
Assume $\Psi_{ij}=\frac{1}{2}\Psi_{ijkl}\psi_k\wedge\psi_l+\Psi_{ijka}\psi_k\wedge\psi_a+\frac{1}{2}\Psi_{ijab}\psi_a\wedge\psi_b,$
$\Psi_{ia}=\frac{1}{2}\Psi_{iakl}\psi_k\wedge\psi_l+\Psi_{iakb}\psi_k\wedge\psi_b+\frac{1}{2}\Psi_{iabc}\psi_b\wedge\psi_c,$
and $\Psi_{ab}=\frac{1}{2}\Psi_{abkl}\psi_k\wedge\psi_l+\Psi_{abkc}\psi_k\wedge\psi_c+\frac{1}{2}\Psi_{abcd}\psi_c\wedge\psi_d.$
According to Proposition \ref{connection metric}, the Ricci curvature  tensor $\mathrm{Ric}^{S_r}$ of $S_r(\xi)$ is given by
\begin{eqnarray}
\mathrm{Ric}^{S_r}_{ij}&=&\Psi_{ikjk}+\Psi_{iaja},\nonumber\\
\mathrm{Ric}^{S_r}_{ia}&=&\Psi_{ikak}+\Psi_{ibab},\nonumber\\
\mathrm{Ric}^{S_r}_{ab}&=&\Psi_{akbk}+\Psi_{acbc},\nonumber
\end{eqnarray}
where
\begin{eqnarray}
\Psi_{ikjk}&=&R_{ikjk}-3A_{ik\alpha}A_{jk\alpha},\nonumber\\
\Psi_{iaja}&=&A_{ik\alpha}A_{jk\beta}P_{\alpha a}P_{\beta a},\nonumber\\
\Psi_{ikak}&=&\frac{1}{2}R_{\alpha\beta ik,k}v_{\beta}P_{\alpha a},\nonumber\\
\Psi_{ibab}&=&0,\nonumber\\
\Psi_{akbk}&=&A_{km\alpha}A_{km\beta}P_{\alpha a}P_{\beta b},\nonumber\\
\Psi_{acbc}&=&\frac{1}{r^2}\delta_{ab}(n-2).\nonumber
\end{eqnarray}
Observe that $A_{ij\alpha}=\frac{1}{2}R_{\alpha\beta ij}v_{\beta}=\frac{1}{2}rR_{\alpha\beta ij}u_{\beta}$, and
the norms of $R_{\alpha\beta ij}$, $R_{\alpha\beta ij,k}$ are bounded on $M$, since $M$ is a closed Riemannian manifold.
By the assumption $M$ has positive Ricci curvature, it follows that, for $r> 0$ small enough,
the Ricci curvature tensor $\mathrm{Ric}^{S_r}$ is positive definite. And the proof is complete.
\end{proof}

\section{Ricci curvature and isoparametric foliation on double manifold}
In this section, we will prove Theorem \ref{main}. The key observation is the following fact in \cite{TXY12}.
\begin{prop}(\cite{TXY12})
Let $\xi$ be a vector bundle of rank $n$ over $M$ of dimension $m$ and $D(\xi, M)$ the double manifold by gluing two copies of
disc bundle of $\xi$. Then $D(\xi, M)$ is diffeomorphic to $S(\xi \oplus \mathbf{1})$, where $S(\xi \oplus \mathbf{1})$
is the sphere bundle of the Whitney sum between $\xi$ and a trivial line bundle $\mathbf{1}$.
\end{prop}
For convenience, denote the total space of $\xi \oplus \mathbf{1}$ by $E(\xi \oplus \mathbf{1})$ and the projection also by $\pi: E(\xi \oplus \mathbf{1})
\rightarrow M$. As in Section 2, we fix a metric $\langle\cdot,\cdot\rangle$ on $\xi$ and choose a connection $D$ compatible with $\langle\cdot,\cdot\rangle$.
At present, for the vector bundle $\xi \oplus \mathbf{1}$, it is natural to extend the metric $\langle\cdot,\cdot\rangle$ on $\xi$
to $\xi \oplus \mathbf{1}$ such that $\xi \oplus \mathbf{1}$ is an orthogonal Whitney sum, denoted also by $\langle\cdot,\cdot\rangle$.
Moreover, choose a unit section $e_{m+n+1}$ for
the trivial line bundle $\mathbf{1}$. By demanding that $e_{m+n+1}$ is parallel, we extend naturally the connection $D$ on $\xi$ to a connection
on $\xi \oplus \mathbf{1}$, denoted also by $D$.

Assume $U$ is small enough such that $\xi|_U$ is trivial, we can choose orthonormal cross sections
$e_{m+1},...,e_{m+n}$ on $U$ for $\xi$. Then $e_{m+1},...,e_{m+n}, e_{m+n+1}$ is an orthonormal cross sections
on $U$ for $\xi \oplus \mathbf{1}$. Moreover, for $m+1\leq \alpha, \beta\leq m+n$, $De_{\alpha}=\omega_{\alpha\beta}e_{\beta},$ and $De_{m+n+1}=0$,
where $(\omega_{\alpha\beta})$ is the matrix expression of the connection $D$. Consequently, $\omega_{m+n+1~\alpha}=-\omega_{\alpha~m+n+1}=0$.

\noindent\textbf{Proof of Theorem \ref{main}:}
\begin{proof}
By the assumption, choose the Riemannian metric $ds^2_M$ on $M$ with positive Ricci curvature. According to Theorem \ref{Positive Ricci}, for the connection metric on $E(\xi \oplus \mathbf{1})$, and for $r_0>0$ small enough, the sphere bundle $S_{r_0}(\xi \oplus \mathbf{1})$ of radius $r_0$ with the induced metric has positive Ricci curvature.

Next, we will construct the natural isoparametric foliation on $S_{r_0}(\xi \oplus \mathbf{1})$. Let $f: S_{r_0}(\xi \oplus \mathbf{1})\rightarrow \mathbb{R}$ be a smooth function such that
$f(p,v)$ is the $e_{m+n+1}$-component of $v$ for each point $(p,v)\in S_{r_0}(\xi \oplus \mathbf{1})$. We will prove that $f$ is an isoparametric function.

Choose an open neighborhood $U$ in $M$ such that $ds_M^2|_U=\sum_{i=1}^{m}(\omega_i)^2$,
where $\omega_1,..., \omega_m$ are $1$-forms on $U$. Let $\omega_{ij}$ and $\Omega_{ij}$ be the connection $1$-forms and
curvature $2$-forms, respectively. Moreover, assume $U$ is small enough such that $\xi|_U$ is trivial, and choose orthonormal cross sections
$e_{m+1},...,e_{m+n}$ for $\xi|_U$, with connection $1$-forms $\omega_{\alpha\beta}$ and curvature 2-forms $\Omega_{\alpha\beta}$. For the vector bundle $\xi \oplus \mathbf{1}$, we assume $e_{m+n+1}$ is parallel as before, i.e., $\omega_{m+n+1~\alpha}=-\omega_{\alpha~m+n+1}=0$.

Under this local trivialization of $(\xi \oplus \mathbf{1})|_U$, for each point $(p, v)\in (\xi \oplus \mathbf{1})|_U$, we have $v=v_{\alpha}e_{\alpha}+v_{m+n+1}e_{m+n+1}$. Define $\theta_i:=\pi^*\omega_i$, $\theta_{\alpha}:=dv_{\alpha}
+v_{\beta}\omega_{\beta\alpha}$ and $\theta_{m+n+1}:=dv_{m+n+1}$. The connection metric $ds^2_{(\xi \oplus \mathbf{1})}$ can be expressed as
$ds^2_{(\xi \oplus \mathbf{1})}=(\theta_i)^2+(\theta_{\alpha})^2+(\theta_{m+n+1})^2$. Write $u_{\alpha}=\frac{v_{\alpha}}{r}$ and $u_{m+n+1}=\frac{v_{m+n+1}}{r}$, where $r^2=(v_{\alpha})^2+(v_{m+n+1})^2$. Then, for the sphere bundle $S_{r_0}(\xi \oplus \mathbf{1})$, the induced metric $ds^2_{S_{r_0}(\xi \oplus \mathbf{1})}$
metric can be expressed as
$$ds^2_{S_{r_0}(\xi \oplus \mathbf{1})}=(\theta_i)^2+r_0^2(du_{\alpha}+u_{\beta}\omega_{\beta\alpha})^2+r_0^2(du_{m+n+1})^2.$$
Under this local coordinate system, $f(p, v)=r_0u_{m+n+1}$ for $(p, v)\in S_{r_0}(\xi \oplus \mathbf{1})$. It follows that $|\nabla f|^2=1-\frac{f^2}{r_0^2}$.

To complete the proof, it is sufficient to show the regular hypersurfaces of $f$ have constant principal curvatures. Write $w_{\alpha}=\frac{u_{\alpha}}{\sqrt{1-(u_{m+n+1})^2}}$, then
$$ds^2_{S_{r_0}(\xi \oplus \mathbf{1})}=(\theta_i)^2+r_0^2\frac{(du_{m+n+1})^2}{1-(u_{m+n+1})^2}+r_0^2(1-(u_{m+n+1})^2)(dw_{\alpha}+w_{\beta}\omega_{\beta\alpha})^2.$$
Using the substitution $t=\arccos{u_{m+n+1}}$ for $-1<u_{m+n+1}<1$, we obtain that
$$ds^2_{S_{r_0}(\xi \oplus \mathbf{1})}=(\theta_i)^2+r_0^2(dt)^2+r_0^2\sin^2{t}(dw_{\alpha}+w_{\beta}\omega_{\beta\alpha})^2.$$
Now, we can choose $\varphi_{i}=\theta_i$, $\varphi_{0}=r_0dt$, and $\varphi_{a}=r_0\sin{t}\psi_{a}$ such that $$\sum_{a=m+1}^{m+n-1}(\psi_{a})^2=(dw_{\alpha}+w_{\beta}\omega_{\beta\alpha})^2.$$
Then $\{\varphi_{0}, \varphi_{i}, \varphi_{a}\}$ is a moving frame for $S_{r_0}(\xi \oplus \mathbf{1})$ with $\varphi_{0}|_{f^{-1}(c)}=0$ for $-r_0<c<r_0$. Let $\varphi_{0i}, \varphi_{0a}, \varphi_{ij}, \varphi_{ia}, \varphi_{ab}$ be the connection $1$-forms. To obtain the principal curvatures of $f^{-1}(c)\subset S_{r_0}(\xi \oplus \mathbf{1})$, it is sufficient to determine the forms $\varphi_{0i}$ and $\varphi_{0a}$. Define
\begin{eqnarray}
\varphi_{0i}&=&h_{ij}\varphi_j+h_{ia}\varphi_a,\nonumber\\
\varphi_{0a}&=&h_{ai}\varphi_i+h_{ab}\varphi_b,\nonumber\\
\varphi_{ij}&=&\lambda_{ij}\varphi_0~\mathrm{mod}\{\varphi_{i}, \varphi_{a}\},\nonumber\\
\varphi_{ia}&=&\lambda_{ia}\varphi_0~\mathrm{mod}\{\varphi_{i}, \varphi_{a}\},\nonumber\\
\varphi_{ab}&=&\lambda_{ab}\varphi_0~\mathrm{mod}\{\varphi_{i}, \varphi_{a}\}.\nonumber
\end{eqnarray}
From the first structural equations, we have
\begin{eqnarray}
d\varphi_{0}&=&\varphi_{0i}\wedge\varphi_{i}+\varphi_{0a}\wedge\varphi_{a},\nonumber\\
d\varphi_{i}&=&\varphi_{i0}\wedge\varphi_{0}+\varphi_{ij}\wedge\varphi_{j}+\varphi_{ia}\wedge\varphi_{a},\nonumber\\
d\varphi_{a}&=&\varphi_{a0}\wedge\varphi_{0}+\varphi_{ai}\wedge\varphi_{i}+\varphi_{ab}\wedge\varphi_{b}.\nonumber
\end{eqnarray}
By a direct computation, we have
\begin{eqnarray}
d\varphi_{0}&=&0,\nonumber\\
d\varphi_{i}&=&\pi^*\omega_{ij}\wedge \varphi_{j},\nonumber\\
d\varphi_{a}&=&\frac{\cot{t}}{r_0}\varphi_0\wedge \varphi_{a}~\mathrm{mod}\{\varphi_{i}\wedge \varphi_{j}, \varphi_{i}\wedge \varphi_{a},
\varphi_{a}\wedge \varphi_{b}\}.\nonumber
\end{eqnarray}
It follows that
\begin{equation}
\left\{ \begin{array}{lll}
h_{ij}=h_{ji}, h_{ia}=h_{a i}, h_{ab}=h_{ba};\nonumber\\
h_{ij}=-\lambda_{ij}, h_{ia}=-\lambda_{ia};\nonumber\\
h_{ai}=\lambda_{ia}, h_{ab}=-\lambda_{ab}+\frac{\cot{t}}{r_0}\delta_{ab}.\nonumber
\end{array}\right.
\end{equation}
Therefore, $\varphi_{0i}=0$ and $\varphi_{0a}=\frac{\cot{t}}{r_0}\varphi_a$. That is to say, the hypersurface $f^{-1}(c)$
has constant principal curvatures $0$ and $-\frac{\cot{t}}{r_0}$ with $t=\arccos{\frac{c}{r_0}}$ for the unit normal vector field
determined by $\varphi_0$.
The proof is complete.
\end{proof}

\section{Applications to isoparametric foliation of unit spheres}
In this section, based on Theorem \ref{main}, we will study the Ricci curvature and isoparametric foliation on double manifolds $D(S^{n+1}_{\pm})$, and prove Theorem \ref{isoparametric case}.

\noindent\textbf{Proof of Theorem \ref{isoparametric case}:}

\begin{proof}
Let $Y^n$ be a closed isoparametric hypersurface with $4$ distinct principal curvatures in $S^{n+1}(1)$ with two focal submanifolds $M_{+}$ and $M_{-}$
of codimension $m_1+1$ and $m_2+1$ in $S^{n+1}(1)$ respectively(cf. \cite{CR85}). And $n=2(m_1+m_2)$. Without loss of generality, we can assume $m_1\leq m_2$.

\noindent \textbf{Case (1). $(m_1, m_2)=(1, k)$:} According to \cite{Ta76}, the isoparametric hypersurfaces in this case must be homogeneous. More precisely, they are
the principal orbits of isotropy representation of the symmetric pair $(\mathrm{SO}(k+4), \mathrm{SO}(2)\times \mathrm{SO}(k+2))$. For the symmetric pair $(\mathrm{SO}(k+4), \mathrm{SO}(2)\times \mathrm{SO}(k+2))$, we have the Cartan decomposition $\mathfrak{o}(k+4)=(\mathfrak{o}(2)+\mathfrak{o}(k+2))\oplus
\mathfrak{p}$, where $\mathfrak{p}$ is the orthogonal complement of $(\mathfrak{o}(2)+\mathfrak{o}(k+2))$ with respect to the Killing form $B$ of $\mathfrak{o}(k+4)$,
and $\mathfrak{p}$ has the canonical induced inner product. Let $S(\mathfrak{p})$ be the unit sphere in $\mathfrak{p}$. Then $G=\mathrm{SO}(2)\times \mathrm{SO}(k+2)$ acts on $\mathfrak{p}$ by conjugation and it induces the cohomogeneity one action on $S(\mathfrak{p})$, the principal orbits of which are isoparametric hypersurfaces in this case. Let $K_+=\Delta \mathrm{SO}(2)\times\mathrm{SO}(k)$, $K_-=\mathbb{Z}_2\cdot\mathrm{SO}(k+1)$ and
$H=\mathbb{Z}_2\cdot\mathrm{SO}(k)$ be the closed subgroups of $G$. Under the notations of \cite{GZ02}, $(S(\mathfrak{p}), G)$ is a chomogeneity one manifold determined by the group diagram $H\subset K_{+}, K_{-} \subset G$.
The two focal submanifolds(singular orbits) $M_+$ and $M_-$ are determined by
$M_+= G/K_+\cong V_{k+2, 2}$, the Stiefel manifold of $2$-orthonormal frames in $\mathbb{R}^{k+2}$, and
$$M_-=G/K_-\cong (S^1\times S^{k+1})/\mathbb{Z}_2=(S^1\times S^{k+1})/(\theta, x)\sim(\theta+\pi, -x).$$
Thus, the double manifold $D(S^{n+1}_+)$ has an induced cohomogeneity one action by $G$ with group diagram $H\subset K_+, K_+\subset G$. By Van Kampen's
theorem, $\pi_1D(S^{n+1}_+)\cong\pi_1 M_+=\mathbb{Z}_2$ if $k=1$, and $\pi_1D(S^{n+1}_+)\cong\pi_1 M_+=\{0\}$ if $k>1$. It follows from Theorem A in \cite{GZ02} that $D(S^{n+1}_+)$ admits an invariant metric with positive Ricci curvature. Meanwhile, the principal orbits are homogeneous isoparametric hypersurfaces. For the case of $D(S^{n+1}_-)$, it has the cohomogeneity one action by $G$ with group diagram $H\subset K_-, K_-\subset G$. By Van Kampen's theorem again, $\pi_1D(S^{n+1}_-)\cong\pi_1 M_-=\mathbb{Z}$. Thus, by Theorem A
in \cite{GZ02}, $D(S^{n+1}_-)$ only admits an invariant metric with nonnegative Ricci curvature, and a homogeneous isoparametric foliation.

%\begin{thm}(\cite{GZ02})
%Assume $M$ is a closed manifold which has a cohomogeneity one action by %compact Lie group $G$. Then $M$ always admits a $G$-invariant metric with %non-negative Ricci curvature. Moreover, $M$ admits a $G$-invariant metric %with positive Ricci curvature if and only if $\pi_1 M$ is finite.
%\end{thm}

\noindent \textbf{Case (2). $2\leq m_1\leq m_2$:}
According to \cite{Mu80}, both of the focal submanifolds are minimal submanifolds in $S^{n+1}(1)$. Moreover, for any unit normal vector $\eta$ of $M_+$( respectively $M_-$ ), the principal curvatures of the shape operator
$A_{\eta}$ are $1, -1, 0$.
To complete the proof, it is sufficient to only consider $M_+$ and prove the induced metric of $M_+$ in $S^{n+1}(1)$ has positive Ricci curvature. Choose a local
adapted moving frame $\{e_{i}, e_{\alpha}\}$ of $M_+$ for $1\leq i\leq m_1+2m_2$ and $m_1+2m_2+1\leq\alpha\leq n+1$, where $e_i$'s are tangent to $M_+$
and $e_{\alpha}$'s are normal to $M_+$. Then, by the Gauss equation, for each unit tangent vector $X$ of $M_{+}$, the Ricci curvature $\mathrm{Ric}^{M_+}$ is given by
$$\mathrm{Ric}^{M_+}(X, X)=2m_2+m_1-1-\sum_{\alpha}\langle A_{\alpha}X, A_{\alpha}X\rangle.$$
For any unit normal vector $\eta$ of $M_+$, the principal curvatures of the shape operator
$A_{\eta}$ are $1, -1, 0$. It follows that $\langle A_{\alpha}X, A_{\alpha}X\rangle\leq 1$. Therefore,
$\mathrm{Ric}^{M_+}(X, X)\geq 2(m_2-1)>0$ if $m_2\geq 2$. Now, the proof is complete.
\end{proof}

\begin{ack}
The authors are very grateful to Professor Zizhou Tang for valuable discussions. The second author is also very grateful to Professor Yoshihiro Ohnita
for helpful conversations when he attended the 8th OCAMI-KNUGRG Joint International Workshop on
Submanifold Geometry and Related Topics at Osaka City University in 2014. The first author was supported in part by the NSFC (No. 11331002, No. 11471299), and the second author was partially supported by the NSFC (No. 11401560 and No. 11571339).
\end{ack}

\end{document}